\documentclass{article}
\usepackage[utf8]{inputenc}
\usepackage{graphicx} 

\usepackage[utf8]{inputenc}
\usepackage[T1]{fontenc}
\usepackage{amsmath}
\usepackage{amssymb}
\usepackage{amsthm}
\usepackage{amsfonts}
\usepackage[english]{babel}
\usepackage{lmodern}
\usepackage{mathrsfs}
\usepackage{csquotes}
\usepackage{enumitem}
\usepackage{array}
\usepackage{amsfonts}
\usepackage{mathtools}
\usepackage{esint}
\usepackage{array}
\usepackage{authblk}
\usepackage{nccmath}
\usepackage{bold-extra}
\usepackage{cite}
\usepackage{tikz}
\usepackage{overpic}
\usepackage{geometry}
\usepackage{placeins}
\usetikzlibrary{hobby}
\usetikzlibrary{backgrounds}

\usepackage{pgf}
\usepackage{pgfplots}
\pgfplotsset{compat=newest}
\usepackage{xcolor}
\usetikzlibrary{plotmarks,backgrounds,patterns}
\usetikzlibrary{external}
\tikzset{external/only named=true}

\usepackage{selinput}

\usepackage{todonotes}

\usepackage{dsfont}

\newtheorem{theorem}{Theorem}[section]
\newtheorem{lemma}[theorem]{Lemma}
\newtheorem{proposition}[theorem]{Proposition}
\newtheorem{corollary}[theorem]{Corollary}
\newtheorem{definition}[theorem]{Definition}

\newtheorem{example}[theorem]{Example}
\theoremstyle{definition}

\newcommand{\RR}{\mathbb{R}}

\newcommand{\NN}{\mathbb{N}}

\newcommand{\norm}[1]{\left\lVert#1\right\rVert}

\title{Regularity Properties of Solutions of a Model for Morphoelastic Growth in the Presence of Nutrients in One Spatial Dimension}
\author{Julian Blawid\footnote{Faculty of Mathematics, University of Regensburg; Regensburg; Germany;\\ e-mail: julian.blawid@mathematik.uni-regensburg.de (corresponding author)}, Georg Dolzmann\footnote{Faculty of Mathematics, University of Regensburg; Regensburg; Germany}}
\date{30.07.2025}

\begin{document}

\maketitle

\begin{abstract}
Regularity properties of solutions for a class of quasi-stationary models in one spatial dimension for stress-modulated growth in the presence of a nutrient field are proven. At a given point in time the configuration of a body after pure growth is determined by means of a family of ordinary differential equations in every point in space. Subsequently, an elastic deformation, which is given by the minimizer of a hyperelastic variational integral, is applied in order to restore Dirichlet boundary conditions. While the ordinary differential equations governing the growth process depend on the elastic stress and the pullback of a nutrient concentration in the current configuration, the hyperelastic variational problem is solved on the intermediate configuration after pure growth. Additionally, the coefficients of the reaction-diffusion equation determining the nutrient concentration in the current configuration depend on the elastic deformation and the deformation due to pure growth.
\end{abstract}

\textbf{Acknowledgement:}
  The first author was supported by the Graduiertenkolleg 2339 IntComSin of the Deutsche Forschungsgemeinschaft (DFG, German Research Foundation) – Project-ID 321821685.

\section{Introduction}

In recent decades, the mathematical community has developed an ever-increasing interest in the modeling of biomechanical phenomena. In particular, the theory of morphoealstic growth, that is, growth under the influence of elastic stress, has sparked significant interest in the past years because of the many challenging problems that arise in the analysis of the proposed models. Although many mathematical models for different kinds of growth phenomena have been developed, for example in \cite{Ben.Amar.Martine.Goriely.2005}, \cite{Goriely.RobertsonTessi.Vandiver.2008}, \cite{Goriely.Vandiver.2009}, \cite{Moulton.Lessinnes.Goriely.2020} and in particular in \cite{Goriely.2017}, rigorous mathematical results for such models are relatively scarce. This is largely because of the fact that the analysis of dynamical or quasistatic models involving finite elasticity poses significant challenges due to a lack of regularity of minimiziers of hyperelastic variational integrals combined with a blow-up of the energy density as the determinant of the deformation gradient approaches zero. In the $n$-dimensional case, this difficulty has to be circumvented by introducing certain regularizations. In \cite{Davoli.Nik.Stefanelli.2023}, existence of solutions was proved via a time-incremental update scheme after a regulatization of the system. Moreover, models in the context of thermoviscoelasticity involving a blow-up of the energy density as the determinant of the deformation gradient approaches zero were treated in \cite{Mielke.Roubicek.2020}, \cite{Badal.Friedrich.Kruzik.2023} and \cite{Badal.Friedrich.Machill.2024}. Because in our model, the change in volume due to growth is modeled by an ordinary differential equation in every point in space it is not possible to recover the deformation due to pure growth in the $n$-dimensional setting. As the elastic deformation happens after an incremental phase of pure growth and should therefore be defined on some intermediate configuration, another difficulty arises. This issue can be partially resolved by introducing a multiplicative decomposition of the deformation gradient into an elastic part and a growth-related part. Models featuring this multiplicative decomposition of the deformation gradient in the context of crystal plasticity, where similar issues arise, were, for example, treated in \cite{Kroener.1960} and \cite{Lee.1969}. Significant analytical advancements for models with a multiplicative decomposition of the deformation gradient in the context of elastoplasticity were obtained in \cite{Mielke.2003} and \cite{Mielke.2004}. For growth processes, this idea has been adapted in \cite{Skalak.Dasgupta.Moss.Otten.Dullemejer.Vilman.1982}, \cite{Rodriguez.Hoger.McCulloch.1994}, \cite{Taber.1995}, \cite{Lubarda.Hoger.2002} and \cite{Menzel.Kuhl.2021} where the deformation of the reference configuration into some virtual, non-existent intermediate configuration due to plastic deformation is replaced by a deformation due to pure, stress-free volumetric growth. In the one-dimensional situation this difficulty does not arise because one can recover the deformation due to pure growth solely from the information about the volumetric change due to pure growth which is determined by the ordinary differential equations simply by integrating. A theory for growth processes based on the split of the full process into stress-free volumetric growth and a subsequent mechanical response has been developed in \cite{Ambrosi.Mollica.2002}, \cite{Ambrosi.Guillou.2007} and \cite{Ambrosi.Ateshian.Arruda.Cowin.Dumais.Goriely.Holzapfel.Humphrey.Kemkemer.Kuhl.Olberding.Taber.Garipikati.2011}. While our model describes growth phenomena in the bulk term, accretive growth such as the thickening of trees has been treated mathematically rigorous in \cite{Davoli.Nik.Stefanelli.Tomasetti.2024}. For free boundary fluid-structure interaction problems in the context of blood-vessel growth a model was proposed in \cite{Yang.Jaeger.Neuss-Radu.Richter.2016} and analytical advancements were achieved in \cite{Abels.Liu.2023.1}, \cite{Abels.Liu.2023.2} and \cite{Abels.Liu.2023.3}.

In this article, we present a regularity result for a one-dimensional model for morphoelastic growth in the presence of nutrient developed by Bangert and Dolzmann in \cite{Bangert.Dolzmann.2023}. In particular, we prove higher differentiability of solutions in time and space which proves to be useful for constructing suitable numerical schemes.

With our model, we would like to describe a one-dimensional elastic material, henceforth referred to by "rod", consisting of elastic tissue which is confined to a straight line and which grows over time. The key assumption is that, in the presence of a nutrient field, the growth process is driven by the internal stresses of the rod. While the rod is allowed to change it's density in every point in space, the confinement to a straight line excludes bending phenomena. In this particular, one-dimensional situation, twisting phenomena which would correspond to a torsion of the rod are possible from a modeling perspective. However, such twisting phenomena are not considered in our model. As we only allow for orientation preserving, invertible deformations, buckling effects are also not considered. A model involving the aforementioned effects is treated in \cite{Bressan.Palladino.Shen.2017}. In its initial configuration, that is, at time $t=0$, the configuration of the rod is given by the interval $I_0:=(0,L_0) \subset \RR$ for some $L_0>0$. We assume that the rod is confined to a certain line-segment at all times, which is realized by imposing Dirichlet boundary conditions. For $T>0$, we denote by $y:[0,T] \times [0,L_0] \to \RR$ the deformation of the rod and we assume that, given $l_0>0$, for all $t \in [0,T]$ the Dirichlet boundary conditions $y(t,0)=0$ and $y(t,L_0)=l_0$ are satisfied. Even though our analysis does not require this, we assume in the following that $l_0=L_0$ to guarantee an unstressed initial configuration. As the growth process induces a volumetric change at a given point in space and time, which is described by the growth tensor $G:[0,T] \times [0,L_0]\to \RR$, the growth map $g:[0,T] \times [0,L_0] \to \RR$, representing the deformation after pure growth, can be recovered in this special one-dimensional situation by means of the relation $\partial_X g(t,\cdot)=G(t,\cdot)$. After this initial phase of pure growth, the rod is described by an intermediate, natural or virtual configuration $I_v(t):=(g(t,0),g(t,L_0))$. 

However, in this virtual configuration, the Dirichlet boundary conditions are violated. To restore them, we introduce a subsequent elastic deformation $\phi:[0,T] \times [0,L_0]\to \RR$. This elastic deformation is computed by minimizing a hyperelastic variational integral on the virtual configuration which leads to the rod being in the current or deformed configuration $I_c:=[0,l_0]$. A constituitive assumption in our model is that the material properties do not change after pure growth. Therefore, it suffices to define a stored energy density $W:[0,L_0] \times (0,\infty) \to [0,\infty)$ in the reference configuration. After growth induced by the growth tensor $G$, the stored energy density on the natural configuration is then, in every $p \in \RR$ and $t \in [0,T]$, given by $W_G(\cdot,p):=W(g^{-1}(t,\cdot),p)$. Under sufficient coercivity, regularity and convexity assumptions on the stored energy density $W$, the elastic deformation $\phi:[g(t,0),g(t,L_0)]\to \RR$ can then be computed by minimizing the hyperelastic variational integral
\begin{equation*}
    \int_{g(t,0}^{g(t,L_0)} W_G(z,\partial_z\phi(t,z))dz=\int_{g(t,0)}^{g(t,L_0)} W(g^{-1}(t,z),\partial_z\phi(t,z))dz
\end{equation*}
on the space of all absolutely continuous functions satisfying $\varphi(t,g(t,0))=0$ and $\varphi(t,g(t,L_0))=l_0$. This elastic deformation induces internal stresses in the material. Even though we assume the stored energy density $W$ to blow up as the differential of the elastic deformation approaches zero, Ball proved in \cite{Ball.1981} that minimizers in fact satisfy the associated Euler-Lagrange equations as their differential is uniformly bounded away from zero. Hence, the stress
\begin{equation*}
    S(t)=(\partial_pW)(g^{-1}(t,z),\partial_z\phi(t,z))
\end{equation*}
is constant across the entire rod which is another particularity of the one-dimensional situation. Our aim is to understand how the growth process is influenced by the elastic stresses. 

In nature, growth requires the presence of a nutrient field. Therefore, we complement our model by an elliptic reaction-diffusion equation describing the diffusion and absorption of nutrients. Because diffusion and absorption of nutrients happen in the current configuration $I_c$, the reaction diffusion equation is also formulated in the current configuration. As diffusion and absorption of nutrients happen on much shorter time scales than the growth process, the time dependence of the nutrient concentration is only implicitly given by means of the diffusion rate $D$ and the absorption rate $\beta$ in the current configuration which depend on time. We assume that for all $t \in [0,T] $ the nutrient concentration $n(t,\cdot)$ in the current configuration is given by the solution of the equation
\begin{equation*}
    \partial_x(D(t,\cdot)\partial_x n(t,\cdot))=-\beta(t,\cdot)n(t,\cdot)
\end{equation*}
subject to the boundary conditions $n(t,0)=n_L$ and $n(t,l_0)=n_R$. Here, the specific choice of the diffusion rate $D$ and the absorption rate $\beta$, depending on space and time, are constituitive assumptions of the model. Again, we exploit the fact that the material properties do not change after pure growth, that is, we assume that the diffusion coefficient $D_v$ and the absorption rate $\beta_v$ in the natural configuration are, for all $(t,z)\in [0,T] \times [g(t,0),g(t,L_0)]$, given by
\begin{equation*}
    D_v(t,z)=(D_0 \circ g^{-1})(t,z) \qquad \textup{ and } \beta_v(t,z)=(\beta_0 \circ g^{-1})(t,z)
\end{equation*}
where $D_0$ and $\beta_0$ denote the diffusion coefficient and reaction rate in the reference configuration respectively. As in \cite{Bangert.Dolzmann.2023}, this leads, by transforming the equation from the deformed or current configuration onto the virtual configuration and assuming that solutions transform by concatenation with $\phi$, to the assumption that, for $(t,x) \in (0,T) \times (0,l_0)$ the coefficients in the deformed configuration are given by
\begin{equation}\label{1assumptionsD}
    D(t, x)=D_0(g^{-1}(t,\phi^{-1}(t, x)))\partial_z\phi(t,\phi^{-1}(t, x))\qquad \textup{ and }
\end{equation}
\begin{equation}\label{1assumpitonsbeta}
    \beta(t, x)=(\partial_z\phi(t, \phi^{-1}(t, x)))^{-1}\beta_0(g^{-1}(t, \phi^{-1}(t,x))).
\end{equation}
If the material is severely compressed, which is the case if the volumetric change is large, that means, if the rod has grown a lot, the associated diffusion coefficient is small and the absorption rate is large. On the other hand, if the material experiences little growth or self-absorption, the associated diffusion coefficient is large and the absorption rate is small. We will assume in the following that $D_0,\beta_0\in W^{k+m}([0,L_0])$ for $k,m \in \NN$ and that there exist constants $D_{min},D_{max},\beta_{min},\beta_{max}\in (0,\infty)$ with $D_0(X)\in [D_{min},D_{max}]$ and $\beta_0(X)\in[\beta_{min},\beta_{max}]$ for all $X\in [0,L_0]$ as this allows us to obtain good elliptic regularity estimates.

It remains to clarify how we obtain the growth tensor $G$, which is the basic variable of our model. We saw in the foregoing discussion that we can recover the growth map $g$ from the growth tensor $G$ by integration. This determines the natural configuration $I_v$ on which we solve the hyperelastic variational problem to determine the elastic deformation $\phi$. The total deformation $y$ of the rod is then given by the concatenation, $y=\phi \circ g$. Once we know the growth map $g$ and the elastic deformation $\phi$ we can find the diffusion coefficient $D$ and the absorption rate $\phi$ in the current configuration which determine the nutrient concentration $n$ in the current configuration by solving the elliptic reaction diffusion equation. The nutrient concentration in Lagrangian coordinates $N$ is given by $N:=n \circ y$. However, we still have to state a constituitive equation to determine the growth tensor $G$ which drives the evolution of the system. We assume that $G$ satisfies the ordinary differential equation
\begin{equation*}
    \dot{G}(t,X)=\mathcal{G}(X,G(t,X),S(t),N(t,X)),\quad G(0,X)=1
\end{equation*}
for each $(t,X)\in (0,T) \times [0,L_0]$ where $\mathcal{G}$ satisfies the assumption (G1)-(G4) made in section \ref{sectionglobalexistence1d}. In particular, the local volumetric change of the rod in a given point in space and time is influenced by the elastic stress and the nutrient concentration. Our goal is to prove global existence and uniqueness of solutions of the following system.
\begin{definition}\label{Definition Chapter Regularity Definition of Solutions}
    Let $k,m \in \NN$ with $k \geq 0$ and $m \geq 1$. A function $G\in C^m([0,T],C^k([0,L_0]))$ determines a solution of the one-dimensional growth problem with initial condition $G_0\in C^{k}([0,L_0])$ with $G_0>0$ on $[0,L_0]$ if $G$ is a solution of the ordinary differential equation
    \begin{equation*}
        \dot{G}=\hat{\mathcal{G}}(G),\quad G(0)=G_0
    \end{equation*}
    in the Banach space $C^k([0,L_0])$. Here, for $X\in [0,L_0]$ and $G\in C^k([0,L_0])$
    \begin{equation*}
        \hat{\mathcal{G}}(G)(X)=\mathcal{G}(X,G(X),S(G),N(G))
    \end{equation*}
    with a suitable function $\mathcal{G}$ satisfying the assumptions (G1)-(G4) made in section \ref{sectionglobalexistence1d}. Moreover, $S=S(G)\in\RR$ denotes the elastic stress due to the elastic deformation and $N=N(G)=n(G)\circ y(G)\in H^{k+1}(0,L_0)$ denotes the nutrient concentration in the reference configuration.
\end{definition}

To obtain global existence of solutions via Picard-Lindel\"of's theorem, it is crucial that we show that the maps $G \mapsto S(G)$ and $G \mapsto N(G)$ are Lipschitz continuous and uniformly bounded on a suitable subset of $C^k([0,L_0])$. Because we want to prove a global existence result we have to control the solutions of the ordinary differential equation by imposing sub- and supersolutions. This also allows us to guarantee that the growth map $g$ and the elastic deformation $y$ are globally invertible and orientation preserving by ensuring that $\partial_X g = G\geq c >0$. 

\begin{definition}\label{defB1}
    Let $G_0\in C^k([0,L_0])$ and $R_1,R_k\in (0,\infty)$ with
    \begin{equation*}
        \inf_{X\in [0,L_0]} G_0(X)-R_1=:\Gamma_0>0\qquad \textup{ and } \qquad\sup_{X\in [0,L_0]} G_0(X)+R_1=:\Gamma_1<\infty.
    \end{equation*}
    Then, we define $B_1=B_{C^0([0,L_0])}(G_0,R_1)\cap B_{C^k([0,L_0])}(G_0,R_k)\subset C^k([0,L_0])$. 
\end{definition}
Every $G\in B_1$ satisfies $\inf G\geq \Gamma_0$ and $\sup G\leq \Gamma_1$ and thus, $g$ with $\partial_X g=G$ is bi-Lipschitz. Moreover, all $G\in B_1$ are uniformly bounded with respect to the $C^k$-norm. We say that a given quantity depends only on the global constants if it depends only on $k$, $m$, $T$, $\Gamma_0$, $\Gamma_1$, $R_1$, $R_k$, $L_0$, $l_0$, $D_{min}$, $D_{max}$, $\beta_{min}$ , $\beta_{max}$ and the elastic energy density $W$.

Although our analysis does not take advantage of this, our model features a multiplicative decomposition of the deformation gradient into a purely elastic part and the growth tensor. Because the total deformation $y: (0,T) \times [0,T] \to \RR$ is given by the concatenation of the deformation due to pure growth $g: (0,T) \times [0,L_0] \to \RR$ and the elastic deformation $\phi:(0,T) \times (g(t,0), g(t,L_0))$, $y=\varphi \circ g$, the decomposition of the deformation gradient into the growth part and the elastic part is given by
\begin{equation*}
    \partial_X y(t,X)=(\partial_z\phi)(t, g(t, X))\partial_Xg(t, X)=F_e(t, X)G(t, X).
\end{equation*}
This multiplicative decomposition plays an important role in the $n$-dimensional theory because in this case, the growth map $g$ cannot necessarily be recovered from the growth tensor $G$ and the virtual configuration $I_v$ does not necessarily exist. However, the multiplicative decomposition of the deformation gradient into an elastic contribution and a growth-related part still allows us to treat the local volumetric change due to growth and elastic deformation separately.

\section{The Hyperelastic Variational Problem and Lipschitz Dependence of the Stresses}\label{Section Chapter Regularity The Hyperelastic Variational Problem, Lipschitz Dependence of the Stress on the Growth Tensor and Time Regularity of the Stresses}

In this section, we consider the classical hyperelastic variational problem and state the precise and minimal assumptions under which a minimizer on $g(t, [0,L_0])=[g(t, 0),g(t, L_0)]$ exists and lies in the space $C^{k+1}([g(t,0),g(t,L_0)])$. Note, that we require $\partial_X \phi$ to be of the class $C^k$ as the elastic stress $S(G)$ associated to a growth tensor $G \in C^k([0,L_0])$ depends on it and we will later prove that the map $G \mapsto S(G)$ is Lipschitz continuous with respect to the $C^k$-norm.

We first show existence of a solution in the space $C^{1}([g(0,t),g(L_0,t)])$, which is given according to \cite{Ball.1981}, and subsequently prove the higher regularity via the bootstrap method. We drop the dependence on $t$ and consider $g=g(X)$ with $g'=\partial_X g=G$ as a function of the independent variable $X$ only. The variational problem can be formulated as follows: Suppose that $[g(0),g(L_0)]$ is an interval with non-empty interior. We aim to minimize the functional
\begin{equation*}
    I_G[\phi]=\int_{g(0)}^{g(L_0)} W_G(z,\phi'(z))dz=\int_0^{L_0} W(X,\phi'(g(X))G(X)dX
\end{equation*}
on the space of admissible functions
\begin{equation*}
    \mathcal{A}_G=\{\phi\in W^{1,1}(g(0),g(L_0)):I_G[\phi]<\infty,\,\,\phi(g(0))=0,\phi(g(L_0))=l_0\}
\end{equation*}
with $l_0>0$ given. Here, $W_G$ denotes the free energy of the body in the natural configuration. Remember, that we assume in the model that the elastic properties of the material in a point $g(X)\in[g(0),g(L_0)]$ are the same as in the point $X\in[0,L_0]$ in the reference configuration. Thus, $W_G(z,\cdot)=W(g^{-1}(z),\cdot)$ for $z\in [g(0),g(L_0)]$. We now state the necessary assumptions on the stored energy density to obtain existence and regularity of minimizers. Assumptions (W1)-(W3) are the same as in \cite{Bangert.Dolzmann.2023} and are sufficient to obtain existence. Assumption (W4) gives rise to higher regularity of the minimizer.
\\

(W1) $W\in C^0([0,L_0]\times(0,\infty);[0,\infty))$, $W(X,1)=0$ and for all $X\in[0,L_0]$ fixed $W(X,\cdot)\in C^1(0,\infty)$ with $\partial_pW\in C^0([0,L_0]\times (0,\infty))$.
\\

(W2) There exists a convex function $\theta:(0,\infty)\to\RR$ with $\theta(p)\nearrow\infty$ as $p\searrow 0$ and $\theta(p)/p\nearrow\infty$ as $p\nearrow \infty$ such that for all fixed $X\in [0,L_0]$ $W(X,p)\geq \theta(p)$ on $(0,\infty)$. Furthermore, for all fixed $X\in [0,L_0]$ $W(X,\cdot)$ is strictly convex.
\\

(W3) For all $\lambda\in (0,\infty)$ $\int_0^{L_0} W(X,\lambda)dX<\infty$. 
\\

(W4) $\partial_p W\in C^{k+m}([0,L_0]\times (0,\infty);\mathbb{R})$ and $\partial_{pp}W>0$ on $[0,L_0] \times (0,\infty)$.
\\

The following Lemma follows from the existence of a uniform lower bound for $W(X,\cdot)$ given by $\theta$ and serves as an a priori estimate.

\begin{lemma}[Lemma 3 in \cite{Bangert.Dolzmann.2023}]\label{LemmaChapter Regularity Coercivity Bounds}
    Suppose that $W$ satisfies (W1)-(W3), that $S\in\RR$, that $E\subset [0,L_0]$ with $E\neq\emptyset$ and that $\pi:E\to(0,\infty)$ such that $\partial_p W(X,\pi(X))=S$ for all $X\in E$. Then, there exist constants $P_0,P_1\in(0,\infty)$ which depend only on $\theta$ and $S$ such that for all $X\in E$ $\pi(X)\in [P_0,P_1]$. Moreover, if $\Sigma_0,\Sigma_1\in\RR$ with $\Sigma_0<\Sigma_1$, $P_0$ and $P_1$ can be chosen uniformly for $S\in [\Sigma_0,\Sigma_1]$ and depend only on $\theta, \Sigma_0$ and $\Sigma_1$.
\end{lemma}

As in \cite{Bangert.Dolzmann.2023} we observe that (W2), that is, the strict convexity, implies that the map $\partial_p W(X,\cdot):(0,\infty)\to\RR$, $p\mapsto \partial_p W(X,p)$ is strictly increasing and bijective. Thus, there exists a uniquely defined function $\pi_0:[0,L_0]\times\RR\to (0,\infty)$ such that for all $X\in[0,L_0]$ and for all $S\in\RR$ the equation $\partial_p W(X,\pi_0(X,S))=S$ holds. Since the formula for the time-dependent diffusion coefficient introduces the term $\phi_G'\circ\phi_G^{-1}$ we need to prove Lipschitz continuity of the map $\pi_0$ and all its derivatives up to order $k$ with respect to both variables as $\phi_G' \circ \phi_G^{-1} = (\phi_G' \circ g) \circ (g^{-1} \circ \phi_G^{-1}) = \pi_0(\cdot, S_G) \circ y_G^{-1}$.

\begin{lemma}\label{1LipschitzCk}
    Suppose that $W$ satisfies (W1)-(W4). Then, for $0\leq l\leq k$ the map
    \begin{equation*}
        \partial_X^l\pi_0:(X,S)\mapsto \partial_X^l\pi_0(X,S)=\partial_X^l(\partial_p W(X,\cdot))^{-1}(S)
    \end{equation*}
    is uniformly separately Lipschitz continuous on compact intervals, that is, for all $S_0,S_1\in\RR$ with $S_0<S_1$ there exist constants $L_{\partial_p W^{-1},X}$ and $L_{\partial_p W^{-1},S}$ which depend only on the derivatives of $\partial_p W$ such that for all $X,\Tilde{X}\in[0,L_0]$ and $S, \Tilde{S}\in[S_0,S_1]$
    \begin{equation*}
        |\partial_X^l\pi_0(X,S)-\partial_X^l\pi_0(\Tilde{X},S)|\leq L_{\partial_p W^{-1},X}|X-\Tilde{X}|
    \end{equation*}
    and
    \begin{equation*}
        |\partial_X^l\pi_0(X,S)-\partial_X^l\pi_0(X,\Tilde{S})|\leq L_{\partial_p W^{-1},S}|S-\Tilde{S}|.
    \end{equation*}
\end{lemma}

\begin{proof}
    For the case $l=0$ this statement is already established in Lemma 4 in \cite{Bangert.Dolzmann.2023}. We first consider $S\in[S_0,S_1]$ fixed and prove the first statement. We show that $\partial_X^{l+1}\pi_0(\cdot,S)$ is uniformly bounded on $[0,L_0]$ by a constant independent of the choice of $S$ which proves Lipschitz continuity by the mean value theorem. Suppose that $(X_0,p_0)\in (0,L_0)\times (0,\infty)$ is a solution of the equation $\partial_p W(X,p)=S$. Since $\partial_p W\in C^{k+1}((0,L_0)\times (0,\infty))$ with $\partial_{pp}W>0$ we can apply the theorem of implicit functions. Thus, there exist an open neighborhood $U$ of $X_0$ in $(0,L_0)$ and an open neighborhood of $p_0$ in $(0,\infty)$ and a $k+1$ times differentiable function $p(\cdot,S):U\to V$ such that on $U\times V$ the equation $\partial_p W(X,p)=S$ holds if and only if $p=p(X,S)$. Because the solution of this equation is uniquely determined, it follows that $p(\cdot,S)=\pi_0(\cdot,S)$ on $U$ and, because $X_0\in(0,L_0)$ is arbitrary, $p(\cdot,S)=\pi_0(\cdot,S)\in C^{k+1}(0,L_0)$. As in \cite{Bangert.Dolzmann.2023}, Lemma \ref{LemmaChapter Regularity Coercivity Bounds} implies with $E=(0,L_0)$ and $\pi(\cdot)=\pi_0(\cdot,S)$ that the function $\pi_0(\cdot,S)$ is uniformly bounded with values in $[P_0,P_1]$. By implicit differentiation we obtain
    that
    \begin{equation*}
        \partial_X\pi_0(X,S)=-\frac{\partial_{Xp}W(X,\pi_0(X,S))}{\partial_{pp}W(X,\pi_0(X,S))}
    \end{equation*}
    and since $\pi_0(\cdot,S)$ is uniformly bounded, there exist a constant $c_0>0$ with $\partial_{pp}W\geq c_0$ on on $[0,L_0]\times [P_0,P_1]$. Therefore, $\partial_X \pi_0(\cdot,S)$ is uniformly bounded on $(0,L_0)$ and these bounds imply that the Cauchy criterion for the existence of the limits $\lim_{X\to 0}\pi_0(X,S)$ and $\lim_{X\to L_0}\pi_0(X,S)$ is satisfied. Consequently, $\partial_X\pi_0(\cdot,S)\in C^{0}([0,L_0])$. The local Lipschitz continuity of $\pi_0(\cdot, S)$ in the first argument follows by the mean value theorem and (W4). Inductively computing the derivatives of order $j$ for $0\leq j\leq l+1$ and arguing as before proves the first assertion. Note that in every step, the denominator is bounded from below by $c_0^{2^j}$ and only derivatives up to order $l$ of $W_p$ appear in the numerator. 

    To prove the second assertion we note that, for fixed $X\in [0,L_0]$ and $S\in\RR$, as shown in the proof of Lemma 4 in \cite{Bangert.Dolzmann.2023},
    \begin{equation}\label{section1additionaltimeregularitypi0}
        \partial_S\pi_0(X,S)=(\partial_{pp}W(X,\pi_0(X,S)))^{-1}
    \end{equation}
    which is uniformly bounded as a function of $S$ on compact intervals.  Due to (W4) and the fact that $\pi_0(\cdot,S)\in C^{k}([0,L_0])$ as shown above, we can compute
    \begin{equation*}
        \partial_{XS}\pi_0(X,S)=-\frac{\partial_{Xpp}W(X,\pi_0(X,S))+\partial_{ppp}W(X,\pi_0(X,S))\partial_X\pi_0(X,S)}{(\partial_{pp}W(X,\pi_0(X,S)))^2}
    \end{equation*}
    which is uniformly bounded for $X\in [0,L_0]$ and $S\in\RR$ on compact intervals. By Schwartz' theorem $\partial_{SX}\pi_0(X,S)=\partial_{XS}\pi_0(X,S)$. Hence, we proved the second assertion for the case $l=1$. We proceed inductively for $2\leq l\leq k$ and argue as above. Note that by the quotient rule the denominator is bounded by $c_0^{2^l}$ in every step and that in the numerator derivatives of $W_p$ of order at most $l+1$ of $W_p$ appear. This proves the second claim.
\end{proof}

Arguing as in \cite{Bangert.Dolzmann.2023}, where Theorems 2 and 4 from \cite{Ball.1981} are applied and using the bootstrap method, we obtain the following theorem.

\begin{proposition}\label{Proposition Chapter Regularity Existence of Minimisers}
    Suppose that $G\in B_1$ where $B_1$ satisfies the assumptions in Definition \ref{defB1} and that $W$ satisfies (W1)-(W4). Then, there exists a unique minimizer $\phi_G$ of $I_G$ in the space $\mathcal{A}_G$. Moreover, $\phi_G\in C^{k+1}([g(0),g(L_0)])$ with $\min_{[g(0),g(L_0)]}\phi_G'>0$ and the Euler-Lagrange equation
    \begin{equation*}
        \frac{\partial}{\partial z} (\partial_p W_G(z,\phi_G'(z)))=0
    \end{equation*}
    holds on $[g(0),g(L_0)]$. In particular, the stress $\partial_p W_G(z,\phi_G'(z))$ is constant on $[g(0),g(L_0)]$.
\end{proposition}

\begin{proof}
    By Corollary 2 in \cite{Bangert.Dolzmann.2023} it remains to prove that $\phi_G\in C^{k+1}([g(0),g(L_0)])$. As in Remark 1 in \cite{Bangert.Dolzmann.2023} we can deduce that if (W1)-(W3) hold for $W$, then also for $W_G$ for $G\in B_1$ satisfying the assumptions in Definition \ref{defB1}. Moreover, (W4) also holds for $W_G$ if it holds for $W$ as $g'=G\geq \Gamma_0>0$ with $g\in C^{k+1}([0,L_0])$. By the implicit function theorem, it follows from the fact that the stress is constant, $(W4)$ and that $\partial_{pp}W_G>0$ that $\phi_G'\in C^{1}([g(0),g(L_0)])$. Hence, differentiating and dividing by $\partial_{pp}W_G(X,\phi_G'(X))$ yields for $X\in [g(0),g(L_0)]$
    \begin{equation}\label{equation bootstrap higher regularity minimizer}
        \phi_G''(X)=-\frac{\partial_{Xp}W_G(X,\phi_G'(X))}{\partial_{pp}W_G(X,\phi_G'(X))}.
    \end{equation}
    Note, that the lower bound on $\partial_{pp}W_G(X,\phi_G'(X))\geq c_0>0$ for $X\in [g(0),g(L_0)]$ as the map $X\mapsto \partial_{pp}W_G(X,\phi_G'(X))$ is continuous and defined on a compact interval. Now, since the right-hand side is again differentiable due to (W4) and the fact that $\phi_G\in C^{2}([g(0),g(L_0)])$, we can deduce that $\phi_G\in C^{3}([g(0),g(L_0)])$. Proceeding inductively and taking into account that derivatives of order at most $k+1$ of $W_G$ appear on the right-hand if we further differentiate \ref{equation bootstrap higher regularity minimizer}, we can deduce that $\phi_G\in C^{(k+1)}([g(0),g(L_0)])$. Note, that in the denominator on the right-hand side, only powers of $\partial_{pp}W_G(X,\phi_G'(X))$ appear which are all uniformly bounded from below by $c_0^{2^k}>0$. The uniqueness follows from the strict convexity of $W_G(z,\cdot)$.
\end{proof}

\begin{lemma}\label{boundednessofminimiserinCkplus1}
    Suppose that $W$ satisfies (W1)-(W4) and that $B_1$ satisfies the assumptions in Definition \ref{defB1}. Then, there exist constants $P_0,P_1,P_2,\Sigma_0,\Sigma_1\in\RR$ with $0<P_0$ such that for $G\in B_1$ and $\phi_G\in C^{k+1}([g(0),g(L_0)])$, denoting the minimizer obtained in Proposition \ref{Proposition Chapter Regularity Existence of Minimisers}, $S_G\in [\Sigma_0,\Sigma_1]$, $\phi_G'\in [P_0,P_1]$ on $[g(0),g(L_0)]$ and $\norm{\phi_G}_{C^{k+1}([g(0),g(L_0)])}\leq P_2$.
\end{lemma}

\begin{proof}
    It follows immediately from Lemma 5 in \cite{Bangert.Dolzmann.2023} that $S_G\in [\Sigma_0,\Sigma_1]$ and that $\phi_G'\in [P_0,P_1]$ on $[g(0),g(L_0)]$. It remains to show that $\norm{\phi_G}_{C^{k+1}([g(0),g(L_0)])}\leq P_2$. As in \cite{Bangert.Dolzmann.2023} the assertion follows from the fact that $\pi(X)=\phi_G'(g(X))$ for $X\in [0,L_0]$ since $g:[0,L_0]\to [g(0),g(L_0)]$ is bijective and according to Lemma \ref{1LipschitzCk} $\norm{\pi}_{C^k([0,L_0] \times \RR)}\leq P_2$.
\end{proof}

\begin{lemma}\label{1lipschitzdependenceofthestressonthegrowthtensor}
    Suppose that $W$ satisfies (W1)-(W4) and that $B_1$ satisfies the assumptions in Definition \ref{defB1}. Then, the map $S:B_1\to\RR$, $G\mapsto S(G)$ is of class $C^1$ and there exist constants $M_{S,G}$ and $L_{S,G}$ which depend only on the global constants such that
    \begin{equation*}
        \norm{S}_{C^0(B_1)}\leq M_{S,G} 
        \qquad \textup{ and } \qquad
        \norm{\frac{\partial S}{\partial G}(G)}_{\mathcal{L}(C^k([0,L_0]);\RR)}\leq L_{S,G}.
    \end{equation*}
    In particular, $S$ is globally Lipschitz continuous on $B_1$ with Lipschitz constant $L_{S,G}$.
\end{lemma}

\begin{proof}
    This Lemma is an immediate consequence of Lemma 6 in \cite{Bangert.Dolzmann.2023} because $C^k([0,L_0])\subset\mathcal{L}^{\infty}([0,L_0])$. Note in particular that the assumptions imposed on $B_1$ in Definition \ref{defB1} are more restrictive than the assumptions imposed on the ball $B_1$ in equation (1.7) in \cite{Bangert.Dolzmann.2023}. 
\end{proof}

\section{Lipschitz Dependence of the Nutrient Concentration}

In order to prove higher differentiability of solutions in the sense of Definition \ref{Definition Chapter Regularity Definition of Solutions} we also have to prove that the function $N$ denoting nutrient concentration pulled back to the reference configuration depends Lipschitz continuously on the growth tensor with respect to the $C^k$-norm. The goal of this section is to show the following lemma.
\begin{lemma}\label{1Lipschitzdependencenutrients}
    Suppose $W$ satisfies (W1)-(W4) and that $B_1$ satisfies the assumptions in Definition \ref{defB1}. The map $N:B_1\to C^k([0,L_0])$, $G\mapsto N(G)=n_G\circ y_G$, is Lipschitz continuous and uniformly bounded, that is, there exist constants $M_{N,G}$ and $L_{N,G}$ depending only on the global constants such that for all $G, G_1, G_2\in B_1$
    \begin{equation*}
        \norm{N(G)}_{C^k([0,L_0])}\leq M_{N,G}
    \qquad \textup{ and } \qquad
        \norm{N(G_1)-N(G_2)}_{C^k([0,L_0])}\leq L_{N,G}\norm{G_1-G_2}_{C^k([0,L_0])}.
    \end{equation*}
\end{lemma}

The domain of $n_G$ is the deformed configuration, that is, the interval $[0,l_0]$. Therefore, we first have to show Lipschitz continuity of the map $y_G=\phi_G\circ g$ in order to control the change of variables in $N_G=n_G\circ y_G$. Furthermore, the assumptions on the diffusion coefficient $D_G$ and the reaction rate $\beta_G$ as in (\ref{1assumptionsD}) and (\ref{1assumpitonsbeta}) give rise to the factor $\phi_G'\circ\phi_G^{-1}$.

\begin{lemma}\label{1LipschitzGtoyG}
    Suppose that $W$ satisfies (W1)-(W4) and that $B_1$ satisfies the assumptions in Definition \ref{defB1}. Then, there exist constants $L_{y,G}$ and $L_{\phi'\circ g,G}$ dependent only on the global constants such that for all $G_1,G_2\in B_1$
    \begin{enumerate}
        \item the map $y:B_1\to C^k([0,L_0])$, $G\mapsto y_G=\phi_G\circ g$ is Lipschitz continuous, that is,
        \begin{equation*}
            \norm{y_1-y_2}_{C^k([0,L_0])}=\norm{\phi_1\circ g_1-\phi_2\circ g_2}_{C^k([0,L_0])}\leq L_{y,G}\norm{G_1-G_2}_{C^k([0,L_0])},
        \end{equation*}
        \item the map $\phi'\circ g:B_1\to C^k([0,L_0])$, $G\mapsto \phi_G'\circ g$ is Lipschitz continuous, that is,
        \begin{equation*}
            \norm{\phi'_1\circ g_1-\phi'_2\circ g_2}_{C^k([0,L_0])}\leq L_{\phi'\circ g,G}\norm{G_1-G_2}_{C^k([0,L_0])}.
        \end{equation*}
    \end{enumerate}
\end{lemma}

\begin{proof}
    According to Lemma \ref{boundednessofminimiserinCkplus1} there exist constants $P_0,P_1,P_2,\Sigma_1,\Sigma_2\in\RR$ such that $\phi'_G\in[P_0,P_1]$, $\norm{\phi'_G}_{C^k([g(0),g(L_0)])}\leq P_2$ and $S_G\in[\Sigma_0,\Sigma_1]$. As in the proof of Lemma 7 in \cite{Bangert.Dolzmann.2023} we can conclude that $(\phi_G'\circ g)(X)=\pi_0(X,S_G)$ for all $X\in [0,L_0]$. Let $0\leq l\leq k$. By Lemma \ref{boundednessofminimiserinCkplus1}, $\partial_X^l\pi_0$ is Lipschitz continuous as a function of $S$ on $[\Sigma_0,\Sigma_1]$. Therefore, by Lemma \ref{boundednessofminimiserinCkplus1},
    \begin{align*}
        \norm{\phi_1'\circ g_1-\phi_2'\circ g_2}_{C^k([0,L_0])}=&\sum_{l=0}^k\sup_{X\in[0,L_0]}|\partial_X^l\pi_0(X,S_1)-\partial_X^l\pi_0(X,S_2)|\\ \leq &C|S_1-S_2|\leq C\norm{G_1-G_2}_{C^k([0,L_0])}
    \end{align*}
    for a constant $C>0$ depending only on the global constants. This proves the second assertion. To prove the first assertion, we note that, as shown in the proof of Lemma 7 in \cite{Bangert.Dolzmann.2023}, $\norm{y_1-y_2}_{C^0([0,L_0])}\leq C\norm{G_1-G_2}_{L^\infty(0,L_0)}$. Furthermore, the chain rule together with the second assertion, by computing the $C^k$-norm of the concatenation explicitly, yield
    \begin{equation*}
        \norm{ y_1-y_2}_{C^k([0,L_0])}=\norm{\phi_1\circ g_1-\phi_2\circ g_2}_{C^k([0,L_0])} \leq  C\norm{G_1-G_2}_{C^k([0,L_0])}
    \end{equation*}
    as the concatenation and product of Lipschitz continuous functions is again Lipschitz continuous.
\end{proof}

The following Lemma provides some preliminary Lipschitz estimates. $1.$ and $3.$ will later be used to obtain sufficient regularity of the coefficients in the reaction-diffusion equation to apply elliptic regularity theory and $2.$ will be used to prove Lipschitz dependence of those coefficients on the growth tensor.
\begin{lemma}\label{1lipschitzof3functions}
    Suppose that $W$ satisfies (W1)-(W4), that $0\leq l\leq k$ and that $B_1$ satisfies the assumptions in Definition \ref{defB1}. Then there exist constants $L_{y,X}$, $L_{\phi'\circ\phi^{-1},x}$ and $L_{\phi'\circ g,X}$ depending only on the global constants such that for all $G\in B_1$, $x_1,x_2\in [0,l_0]$ and $X_1,X_2\in [0,L_0]$
    \begin{enumerate}
        \item $|y_G^{(l)}(X_1)-y_G^{(l)}(X_2)|\leq L_{y,X}|X_1-X_2|$,
        \item $|(\phi_G'\circ\phi_G^{-1})^{(l)}(x_1)-(\phi_G'\circ\phi_G^{-1})^{(l)}(x_2)|\leq L_{\phi'\circ\phi^{-1},x}|x_1-x_2|$ and
        \item $|(\phi_G'\circ g)^{(l)}(X_1)-(\phi_G'\circ g)^{(l)}(X_2)|\leq L_{\phi'\circ g,X}|X_1-X_2|$.
    \end{enumerate}
\end{lemma}

\begin{proof}
    By Lemma \ref{boundednessofminimiserinCkplus1} there exist $P_2,\Sigma_1,\Sigma_2\in\RR$ such that $\norm{\phi'_G}_{C^k([g(0),g(L_0)])}\leq P_2$ and $S_G\in[\Sigma_0,\Sigma_1]$. Since for all $X\in [0,L_0]$ $g'(X)=G(X)$ with $G\in B_1$ both $g$ and $\phi_G$ are Lipschitz continuous and bounded with respect to the $C^k$-norm and hence, the map $X\mapsto y^{(l)}_G(X)$ is also Lipschitz continuous with Lipschitz constant $L_{y,X}$ depending only on the global constants. To prove the third claim, we remind ourselves that we have deduced from the Euler Lagrange equations that $(\phi'_G\circ g)(X)=\pi_0(X,S_G)$ and therefore, $(\phi'_G\circ g)^{(l)}(X)=\pi_0^{(l)}(X,S_G)$ which is Lipschitz continuous in $X$ according to Lemma \ref{1LipschitzCk} with uniformly bounded Lipschitz constant for $S\in[\Sigma_0,\Sigma_1]$. To prove the second claim, we refer to \cite{Bangert.Dolzmann.2023} for the case $l=0$ and compute for $1\leq l\leq k$
    \begin{equation*}
        |(\phi_G'\circ\phi_G^{-1})^{(l)}(x_1)-(\phi_G'\circ\phi_G^{-1})^{(l)}(x_2)| =|(\phi_G'\circ g\circ y_G^{-1})^{(l)}(x_1)-(\phi_G'\circ g\circ y_G^{-1})^{(l)}(x_2)|\leq C|x_1-x_2|.
    \end{equation*}
    By the Lipschitz continuity of the map $x\mapsto y_G^{-1}$, which follows from the first assertion and the fact that $y_G'$ is bounded from below by a constant strictly larger than $0$ as shown in \cite{Bangert.Dolzmann.2023},
    the last inequality follows.
\end{proof}

To prove the following lemma, it is sufficient to assume that $D_0,\beta_0\in W^{k+1,\infty}([0,L_0])$. However, in view of proving higher time regularity of the solutions of the elliptic reaction-diffusion equation, we already make the stronger assumption that $D_0,\beta_0\in W^{k+m,\infty}([0,L_0])$.

\begin{lemma}\label{regularitycoefficients1d}
    Suppose that $W$ satisfies (W1)-(W4), that $D_0,\beta_0\in W^{k+m,\infty}([0,L_0])$ and that $B_1$ satisfies the assumptions in Definition \ref{defB1}. Then, the maps
    \begin{equation*}
        D:B_1\to C^k([0,l_0]),\quad G\mapsto D_G=(\phi_G'\circ\phi_G^{-1})(D_0\circ y_G^{-1})
    \end{equation*}
    and
    \begin{equation*}
        \beta:B_1\to C^k([0,l_0]),\quad G\mapsto \beta_G=(\phi_G'\circ\phi_G^{-1})^{-1}(\beta_0\circ y_G^{-1})
    \end{equation*}
    are bounded and Lipschitz continuous. In particular, there exist constants $L_{D,G}$ and $L_{\beta,G}$ depending only on the global constants such that
    \begin{equation*}
        \norm{D_1-D_2}_{C^k{[0,l_0]}}\leq L_{D,G}\norm{G_1-G_2}_{C^k{[0,L_0]}}
    \end{equation*}
    and
    \begin{equation*}
        \norm{\beta_1-\beta_2}_{C^k{[0,l_0]}}\leq L_{\beta,G}\norm{G_1-G_2}_{C^k{[0,L_0]}}.
    \end{equation*}
    Moreover, the coefficients $D_G,\beta_G$ are uniformly elliptic and $D_G$ and $\beta_G$ are uniformly bounded in $W^{k+1,\infty}$ for all $G\in B_1$ with Lipschitz constants independent of $G$.
\end{lemma}

\begin{proof}
    The uniform ellipticity follows as in Lemma 10 in \cite{Bangert.Dolzmann.2023}. We first show the Lipschitz continuity of the map $G\mapsto \phi'_G\circ \phi^{-1}_G$. By Lemma \ref{boundednessofminimiserinCkplus1} $\phi_G'$ is bounded in the $C^k$-norm on $[g(0,g(L_0)]$. We obtain that for all $x\in [0,l_0]$ and $0\leq l \leq k$
    \begin{align*}
        &|(\phi_1'\circ\phi_1^{-1})^{(l)}(x)-(\phi_2'\circ\phi_2^{-1})^{(l)}(x)|\\ \leq &|((\phi_1'\circ g_1)(y_1^{-1}(x))^{(l)}-((\phi_1'\circ g_1)(y_2^{-1}(x))^{(l)}|+|((\phi_1'\circ g_1)(y_2^{-1}(x))^{(l)}-((\phi_2'\circ y_2)(y_2^{-1}(x))^{(l)}|.
    \end{align*}
    As shown in the proof of Lemma \ref{1lipschitzof3functions}, $y_G^{-1}$ is bounded by a constant which is independent of $G$ and Lipschitz continuous with respect to the $C^{k}$-norm on $[0,l_0]$ with a Lipschitz constant which is also independent of $G$. By the rule for computing the derivative of the inverse function and the quotient rule this is also true for products of derivatives of order up to $k$ counting by multiplicity and hence, we can estimate the first term for all $x \in [0,l_0]$ by
    \begin{equation*}
        |((\phi_1'\circ g_1)(y_1^{-1}(x))^{(l)}-((\phi_1'\circ g_1)(y_2^{-1}(x))^{(l)}| \leq  C\norm{G_1-G_2}_{C^l([0,L_0])}
    \end{equation*}
    where the constant $C$ does not depend on $x$. According to Lemma \ref{1lipschitzof3functions}, for all $x \in [0,l_0]$ the second term can be estimated by
    \begin{equation*}
        |((\phi_1'\circ g_1)(y_2^{-1}(x))^{(l)}-((\phi_2'\circ y_2)(y_2^{-1}(x))^{(l)}|\leq C\norm{\phi_1'\circ g_1-\phi_2'\circ g_2}_{C^l([0,l_0])} \leq C\norm{G_1-G_2}_{C^l([0,L_0])}.
    \end{equation*}
    Also here, the constant $C$ does not depend on $x$. Therefore, the Lipschitz continuity of the map $G\mapsto \phi'_G\circ \phi^{-1}_G$ defined on $B_1$ with a Lipschitz constant only depending on the global constants in $C^k([0,l_0])$ follows by taking the supremum and summing up over $0\leq l\leq k$. Moreover, the map $G\mapsto D_0\circ y_G^{-1}$ can by the same argument also be estimated by
    \begin{equation*}
        \norm{D_0\circ y_1^{-1}-D_0\circ y_2^{-1}}_{C^k([0,l_0])}\leq C\norm{G_1-G_2}_{C^k([0,L_0])}
    \end{equation*}
    for some suitable constant $C$ only dependent on the global constants. The Lipschitz continuity of the map $G\mapsto D_G$ with a Lipschitz constant dependent only on the global constants follows as this map is given by the product of two bounded and Lipschitz continuous maps with a Lipschitz constant dependent only on the global constant as shown above. To prove Lipschitz continuity of the map $G\mapsto \beta_G$, note that, as in \cite{Bangert.Dolzmann.2023},
    \begin{equation*}
        \frac{1}{\phi_1'\circ \phi_1^{-1}}-\frac{1}{\phi_2'\circ \phi_2^{-1}}=\frac{\phi_2'\circ \phi_2^{-1}-\phi_1'\circ \phi_1^{-1}}{(\phi_1'\circ \phi_1^{-1})(\phi_2'\circ \phi_2^{-1})}.
    \end{equation*}
    Since the denominator is bounded from below by $P_0^2$, the map $(\phi_G'\circ\phi_G^{-1})^{-1}:B_1\to C^0([0,l_0])$, $G\mapsto (\phi_G'\circ\phi_G^{-1})^{-1}$ is Lipschitz continuous. Also here, the Lipschitz constant depends only on the global constants. The Lipschitz continuity of the map $G\mapsto \beta_G$ then follows analogously to the Lipschitz continuity of the map $G\mapsto D_G$. The Lipschitz continuity of the maps $x\mapsto D_G(x)$ and $x\mapsto \beta_G(x)$ with Lipschitz constants independent of $G$ is an immediate consequence of Lemma 9 in \cite{Bangert.Dolzmann.2023} taking into consideration that $D_0,\beta_0\in W^{k+1,\infty}([0,L_0])$.
\end{proof}

\begin{lemma}\label{1regularityoncoefficients}
    Suppose that $D_0,\beta_0\in W^{k+m,\infty}([0,L_0])$, that $W$ satisfies (W1)-(W4) and that $B_1$ satisfies the assumptions in Definition \ref{defB1}. Let $G\in B_1$. Then, there exists a unique solution $n_G\in H^1(0,l_0)$ of the reaction-diffusion equation induced by $G\in B_1$ given by
    \begin{equation*}
        \begin{cases}
        -(D_Gn'_G)'+\beta_Gn_G=0\quad\text{in}\,\, [0,l_0]\\n_G(0)=n_L\\n_G(l_0)=n_R.
        \end{cases}
    \end{equation*}
    Moreover, there exists a constant $M_{n,H^{k+2}}>0$ such that $n_G$ satisfies the a priori bound
    \begin{equation*}
        \norm{n_G}_{H^{k+2}(0,l_0)}\leq M_{n,H^{k+2}}
    \end{equation*}
    and, if $n_L,n_R\geq 0$, $n_G\geq 0$ on $[0,l_0]$. 
\end{lemma}

\begin{proof}
    The existence of a unique solution and the fact that $n_G\geq 0$ on $[0,l_0]$ if $n_L,n_R\geq 0$ are already proven in Lemma 10 in \cite{Bangert.Dolzmann.2023}. The additional regularity and the uniform bound in the $H^{k+2}$-norm follow from elliptic regularity theory and Lemma \ref{regularitycoefficients1d}. Note, that in particular, due to Lemma \ref{regularitycoefficients1d}, $D_G,\beta_G\in W^{k+1,\infty}([0,L_0])$.
\end{proof}

We are now able to prove Lemma \ref{1Lipschitzdependencenutrients}.

\begin{proof}
    The uniform boundedness follows immediately from Lemma \ref{1regularityoncoefficients}. Suppose that $G_1,G_2\in B_1$ and that $n_1,n_2\in W^{1,2}([0,l_0])$ are the unique weak solutions of the reaction-diffusion equation with diffusion coefficients $D_1, D_2$ and absorption rates $\beta_1,\beta_2$ respectively. By Lemma \ref{1regularityoncoefficients} $n_1,n_2\in H^{k+2}(0,l_0)$ with uniform bounds on the $H^{k+2}$-norm independent of $G\in B_1$. 

    We first verify Lipschitz continuity of the map $n:C^k([0,L_0])\to H^{k+2}(0,l_0)$, $G\mapsto n(G)$. Due to the Sobolev embedding theorem, $n_1-n_2\in C^{k+1}([0,l_0])$. Thus, we can find a polynomial of order $2(k+1)+1$ which we denote by $P_{n_1-n_2}^{k+1}$ such that all derivatives of the map $x\mapsto n_1(x)-n_2(x)-P_{n_1-n_2}^{k+1}(x)$ defined on $[0,l_0]$ up to order $k+1$ vanish in the points $0$ and $l_0$. The construction of such a polynomial is analogous to the construction of the Taylor polynomial. Note, that we have $2(k+1)+2$ unknowns, therefore, a suitable polynomial must have degree $2(k+1)+1$. Proceeding inductively, we assume the statement of the Lemma we are currently proving holds for $l-1$ for $1\leq l\leq k+1$. The case $l=1$ was already proven in \cite{Bangert.Dolzmann.2023}. We differentiate the reaction-diffusion equation $l$-times and test with $n_1^{(l)}-n_2^{(l)}-(P_{n_1-n_2}^{k+1})^{(l)}$ and obtain for $i\in\{1,2\}$ after partial integration that
    \begin{equation*}
        \int_0^{l_0} (D_i n_i')^{(l)}(n_1^{(l+1)}-n_2^{(l+1)}-(P_{n_1-n_2}^{k+1})^{(l+1)})+(\beta_i n_i)^{(l)}(n_1^{(l)}-n_2^{(l)}-(P_{n_1-n_2}^{k+1})^{(l)})dx=0.
    \end{equation*}
    The difference of the two equations is given by
    \begin{align*}
        &\int_0^{l_0}(D_1 n_1'-D_2 n_2')^{(l)}(n_1^{(l+1)}-n_2^{(l+1)}-(P_{n_1-n_2}^{k+1})^{(l+1)}) \\+&(\beta_1 n_1-\beta_2 n_2)^{(l)}(n_1^{(l)}-n_2^{(l)}-(P_{n_1-n_2}^{k+1})^{(l)})dx=0.
    \end{align*}
    For the first addend, we obtain
    \begin{align*}
        &\int_0^{l_0}(D_1 n_1'-D_2 n_2')^{(l)}(n_1^{(l+1)}-n_2^{(l+1)}-(P_{n_1-n_2}^{k+1})^{(l+1)})dx\\ =& \int_0^{l_0} (D_1 n_1^{(l+1)}-D_2 n_2^{(l +1)})(n_1^{(l+1)}-n_2^{(l+1)}-(P_{n_1-n_2}^{k+1})^{(l+1)})\\& +\sum_{\gamma=0}^{l-1} \binom{l}{\gamma} (D_1^{(l-\gamma)}n_1^{(\gamma+1)}-D_2^{(l-\gamma)}n_2^{(\gamma+1)})(n_1^{(l+1)}-n_2^{(l+1)}-(P_{n_1-n_2}^{k+1})^{(l+1)})dx\\ = &\int_0^{l_0} D_1(n_1^{(l + 1)}-n_2^{(l + 1)})^2-(D_1-D_2)n_2^{(l+1)}(n_1^{(l + 1)}-n_2^{(l + 1)})\\&+(D_1 n_1^{(l+1)}-D_2 n_2^{(l +1)})(P_{n_1-n_2}^{k+1})^{(l+1)}\\& +\sum_{\gamma=0}^{l-1} \binom{l}{\gamma} (D_1^{(l-\gamma)}n_1^{(\gamma+1)}-D_2^{(l-\gamma)}n_2^{(\gamma+1)})(n_1^{(l+1)}-n_2^{(l+1)}-(P_{n_1-n_2}^{k+1})^{(l+1)})dx.
    \end{align*}
    The second addend can be treated analogously and we obtain
    \begin{align*}
        &\int_0^{l_0}(\beta_1 n_1-\beta_2 n_2)^{(l)}(n_1^{(l)}-n_2^{(l)}-(P_{n_1-n_2}^{k+1})^{(l)})dx\\=&\int_0^{l_0} \beta_1(n_1^{(l)}-n_2^{(l)})^2-(\beta_1-\beta_2)n_2^{(l)}(n_1^{(l)}-n_2^{(l)})+(\beta_1 n_1^{(l)}-\beta_2 n_2^{(l)})(P_{n_1-n_2}^{k+1})^{(l)}\\& +\sum_{\gamma=0}^{l-1} \binom{l}{\gamma} (\beta_1^{(l-\gamma)}n_1^{(\gamma)}-\beta_2^{(l-\gamma)}n_2^{(\gamma)})(n_1^{(l)}-n_2^{(l)}-(P_{n_1-n_2}^{k+1})^{(l)})dx.
    \end{align*}
    From the construction of $P_{n_1-n_2}^{k+1}$ and Lemma \ref{1regularityoncoefficients} it is immediately clear that this polynomial and all its derivatives are bounded on $[0,l_0]$ by a constant independent of $G$ and hence, we obtain, proceeding as in Lemma 11 in \cite{Bangert.Dolzmann.2023}, for a constant $C>0$ depending only on the global constants such that
    \begin{equation*}
        \norm{n_1-n_2}_{H^{k+1}(0,l_0)}\leq C\norm{D_1-D_2}_{C^{k}(0,l_0)}\norm{n_2'}_{H^k(0,l_0)}+C\norm{\beta_1-\beta_2}_{C^{k}(0,l_0)}\norm{n_2}_{H^k(0,l_0)}.
    \end{equation*}
    By Lemma \ref{1regularityoncoefficients} the Lipschitz continuity of the map $G\mapsto n_G$ in $H^{k+1}(0,l_0)$ follows. By the Sobolev embedding we can conclude that $n_G$ is Lipschitz continuous with respect to the $C^{k}$-norm. The Lipschitz continuity of $N_G$ in $G$ with respect to the $C^{k}$-norm follows directly by the following computation. We remark that $\phi'_G$ as well as $g'=G$ and $n_G$ are uniformly bounded on $B_1$ with respect to the $C^k$-norm by a constant that just depends on the global constants on their respective domain and hence, the function $y_G$ is bounded with respect to the $C^k$-norm. Let $X\in [0,L_0]$ and let $1\leq l\leq k$. Note, that the case $l=0$ was already shown in Lemma 11 in \cite{Bangert.Dolzmann.2023}. Using lemmas \ref{1LipschitzGtoyG} and \ref{1Lipschitzdependencenutrients} and computing the derivative of order $l$ yields
    \begin{align*}
        &|N_1^{(l)}(X)-N_2^{(l)}(X)|=|(n_1\circ y_1)^{(l)}(X)-(n_2\circ y_2)^{(l)}(X)|\\ \leq &|(n_1\circ y_1)^{(l)}(X)-(n_1\circ y_2)^{(l)}(X)|+|(n_1\circ y_2)^{(l)}(X)-(n_2\circ y_2)^{(l)}(X)| \\ \leq &C(\norm{n_1'}_{C^{l}([0,l_0]}\norm{y_1-y_2}_{C^0([0,L_0])}+\norm{G_1-G_2}_{C^{l}([0,L_0])}+\norm{n_1-n_2}_{C^{l}([0,l_0])})\\ \leq& C\norm{G_1-G_2}_{C^{l}([0,L_0])}.
    \end{align*}
    We take the supremum in $X$. The assertion follows by the Sobolev embedding theorem and by summing up over all $l$ with $0\leq l\leq k$.     
\end{proof}

\section{Global Existence of Solutions with Higher Regularity in Space}\label{sectionglobalexistence1d}

In the next step, we prove existence of global in time solutions of the system of ordinary differential equations in the space $C^1([0,T];C^k([0,L_0]))$. The existence of global solutions requires the existence of certain comparison functions to construct global sub- and supersolutions which prevent, from a modeling perspective, arbitrarily fast growth or self-absorption. For $0\leq l\leq k$ we define
\begin{equation*}
    C^l_+([0,L_0]):=\{G\in C^l([0,L_0]): G(X)>0\,\,\text{for all } X\in[0,L_0]\} \qquad\textup{ and}
\end{equation*}
\begin{equation*}
    \mathcal{L}^{\infty}_+([0,L_0]):=\{G\in \mathcal{L}^{\infty}([0,L_0]): G(X)>0\,\,\text{for all } X\in[0,L_0]\}
\end{equation*}
We define the autonomous ordinary differential equation
\begin{equation*}
    G'(t)=\hat{\mathcal{G}}(G(t)),\quad G(0)=1
\end{equation*}
in $C^k([0,L_0])$ where $B_1$ satisfies the assumptions in Definition \ref{defB1} and $\hat{\mathcal{G}}:B_1\to C^k([0,L_0])$ is defined via
\begin{equation*}
    \hat{\mathcal{G}}(G(X)=\mathcal{G}(G(X),S(G),N(G)(X),X).
\end{equation*}
We assume that the constitutive function $\mathcal{G}$ satisfies the following conditions:
\\ 

(G1)
There exist locally Lipschitz continuous functions $\hat{\mathcal{G}}_0,\hat{\mathcal{G}}_1: \mathcal{L}^\infty_+([0,L_0])\to \mathcal{L}^\infty([0,L_0])$ such that $\hat{\mathcal{G}}_i(G)(X)=\mathcal{G}_i(G(X),X)$ for locally Lipschitz continuous functions $\mathcal{G}_i:(0,\infty)\times [0,L_0]\to\RR$ such that for $i\in\{1,2\}$ the ordinary differential equations
\begin{equation*}
    \dot{G}_i(t)=\hat{\mathcal{G}}_i(G_i(t)),\quad G_i(0)=1 
\end{equation*}
have global solutions $G_i\in C^1([0,L_0];\mathcal{L}^\infty([0,L_0]))$ with 
\begin{equation*}
    G_0(t,X)<G_1(t,X) \quad  \textup{ for all } X\in [0,L_0],
\end{equation*}
\begin{equation}\label{equation lower bound global existence chapter regularity 1d}
    G_{min}=\inf_{[0,L_0]}\inf_{[0,T]} G_0(t,X)>0
    \qquad \textup{ and } \qquad
    G_{max}=\sup_{[0,L_0]}\sup_{[0,T]} G_0(t,X)<\infty.
\end{equation}

(G2) The function $\mathcal{G}:\RR_{>0}\times\RR\times\RR_{\geq 0}\times [0,L_0]\to\RR$ has the property that its mixed derivatives up to order $k+m-1$ in its first and third, $m-1$ in its second and $k$ in its fourth variable exist and are continuous. Additionally, $\mathcal{G}$ is jointly continuous in its first three arguments for all $X\in [0,L_0]$ fixed, which follows directly if $m>1$. 
\\

Note that $\mathcal{G}$ is measurable for all $(G,S,N)\in\RR_{>0}\times\RR\times\RR_{\geq 0}$ fixed, which follows automatically since $k\geq 1$.
\\

(G3) For all $\Gamma_i, \Sigma_i, H_i\in\RR$, $i=1,2$, with $0<\Gamma_0<\Gamma_1$, $\Sigma_0<\Sigma_1$, $0\leq H_0<H_1$ and for all $X\in [0,L_0]$ the function
\begin{equation*}
    \mathcal{G}(\cdot,\cdot,\cdot,X):[\Gamma_0,\Gamma_1]\times [\Sigma_0,\Sigma_1]\times [H_0,H_1]\to\RR,\quad (G,S,N)\mapsto \mathcal{G}(G,S,N,X)
\end{equation*}
and its mixed partial derivatives up to order $k$ in the first, third and fourth variable variable are uniformly Lipschitz continuous and bounded with constants $L_\mathcal{G}$ and $M_{\mathcal{G}}$ and the constants are independent of $X$.
\\

(G4) For all $G>0$, $S\in \RR$, $N\geq 0$, $X\in[0,L_0]$
\begin{equation*}
    \mathcal{G}_0(G,X)<\mathcal{G}(G,S,N,X)<\mathcal{G}_1(G,X).
\end{equation*}

Note that by Lemma \ref{1regularityoncoefficients} the nutrient concentration $n_G$ is continuous and nonnegative and that the pointwise evaluation of $N_G$ is justified by the Sobolev embedding theorem. We are now able to prove global existence of solutions of the system.

\begin{proposition}\label{maintheoremsection1dtimespaceregularity}
    Suppose that $T>0$, that $W$ satisfies (W1)-(W4) and that $\hat{\mathcal{G}}$ satisfies (G1)-(G4). Then, there exist a unique solution of the growth problem in the sense of Definition \ref{Definition Chapter Regularity Definition of Solutions} for the case $m=1$.
\end{proposition}

\begin{proof}
    We show the existence of local solutions of the growth problem in the sense of Definition \ref{Definition Chapter Regularity Definition of Solutions} on $[0,T]$ for $m=1$, that is, in the space $C^1([0,T];C^k([0,L_0]))$. First, we will show local existence with the local existence theorem of Picard and Lindel\"of. Then, we will show the existence of global solutions by taking into account that, according to assumption (G1), sub- and supersolutions exist. The existence of global sub- and supersolutions implies that the solution is contained in a closed ball $B_0\subset C^k([0,L_0])$ for $t\in [0,T]$. In view of the lower bound in \ref{equation lower bound global existence chapter regularity 1d} there exists a closed ball $B_1\subset C^k_+([0,L_0])$ with the same center but a greater radius than $B_0$. Because all Lipschitz constants are uniformly bounded on $B_1$ we obtain a uniform bound on the local time of existence due to Picard-Lindel\"of's theorem where we use $G_0\in B_0$ as the initial condition.
    
    We first apply Picard-Lindel\"of's theorem to obtain local existence of solutions of the system. By (G1), the solution of the ordinary differential equations $G_i=\hat{\mathcal{G}}_i(G_i)$, $G_i(0)=1$ with $i\in\{1,2\}$ exists in the Banach space $C^1([0,T];\mathcal{L}^{\infty}([0,L_0]))$ and satisfies
    \begin{equation*}
    G_{min}=\inf_{[0,L_0]}\inf_{[0,T]} G_0(t,X)>0
    \qquad \textup{ and } \qquad
    G_{max}=\sup_{[0,L_0]}\sup_{[0,T]} G_0(t,X)<\infty.
\end{equation*}
For $R_k>0$ yet to be chosen we define the closed balls
\begin{equation*}
    B^0_0=\overline{B}_{C^0}\biggl(\frac{G_{min}+G_{max}}{2}, \frac{G_{max}-G_{min}}{2}\biggr),
\qquad
    B_0^k=\overline{B}_{C^k}\biggl(\frac{G_{min}+G_{max}}{2}, R_k \biggr) \qquad \textup{ and }
\end{equation*}    
\begin{equation*}
    B^0_1=\overline{B}_{C^0}\biggl(\frac{G_{min}+G_{max}}{2}, \frac{G_{max}}{2}\biggr).
\end{equation*}
Clearly, $B_0^0\subset B_1^0\subset C^0_+([0,L_0])$. Moreover, for every $G\in B^0_0$, $\overline{B}_{C^0}(G, G_{min}/2)\subset B^0_1$ and $B_1=B_0^k\cap B_1^0$ satisfies the assumptions in Definition \ref{defB1}. Assumption (G2) ensures that $\mathcal{G}$ is a Caratheodory function. For $G\in B_1$ the function $N(G)\in C^k([0,L_0])$ is continuous and measurable and therefore, the function $X\mapsto \mathcal{G}(G(X),S(G),N(G)(X),X)$ is measurable. 

In the next step, we verify the assumptions of Picard-Lindel\"of's theorem. Let $t_0=0$ and $G_0\in B_0^0$. By Lemma \ref{boundednessofminimiserinCkplus1} applied to $B_1$ there exist $P_0,P_1,P_2,\Sigma_0,\Sigma_1\in\RR$ with $P_0>0$ such that for all $G\in B_1$ the associated stress $S(G)\in[\Sigma_1,\Sigma_2]$ is uniformly bounded and the deformation gradient $\phi'_G$ is uniformly bounded from below by $P_0$ and $\norm{\phi'_G}_{C^k([g(0),g(L_0)])}$ is uniformly bounded by $P_2$. By Lemma \ref{1lipschitzdependenceofthestressonthegrowthtensor} the map $S:B_1\to\RR$, $G\mapsto S(G)$ is Lipschitz continuous and bounded with constants $L_{S,G}$ and $M_{S,G}$. By Lemma \ref{1Lipschitzdependencenutrients} the map $N:B_1\to C^k([0,L_0])$, $N\mapsto N_G$ is Lipschitz continuous and bounded with constants $L_{N,G}$ and $M_{N,G}$. Moreover, by Lemma \ref{1Lipschitzdependencenutrients} we can choose $H_0=0$ and a constant $H_1$ which only depends on the data of the problem with $N_G\geq H_0$ and $\norm{N_G}_{C^k([0,L_0])}\leq H_1$. If we choose $\Gamma_0=G_{min}/2$ and $\Gamma_1= G_{max}+G_{min}/2$, $B_1$ satisfies the assumptions in Definition \ref{defB1}. With these choices of $\Gamma_1,\Gamma_2,\Sigma_0,\Sigma_1, H_0$ and $H_1$ assumption (G3) guarantees that $\mathcal{G}$ is uniformly Lipschitz continuous in the first three arguments and bounded on $[\Gamma_0,\Gamma_1]\times [\Sigma_0,\Sigma_1]\times [H_0,H_1]$ independently of $X$. The right-hand side of the ordinary differential equation is given by
\begin{equation*}
    \hat{\mathcal{G}}(G)(X)=\mathcal{G}(G(X),S(G),N(G)(X),X).
\end{equation*}
Moreover, taking into account that the product of bounded Lipschitz continuous functions is again Lipschitz continuous and (G3), $\hat{\mathcal{G}}$ satisfies for all $G_1,G_2\in B_1$
\begin{equation*}
    \norm{\hat{\mathcal{G}}(G_1)-\hat{\mathcal{G}}(G_2)}_{C^k{[0,L_0]}} \leq L_{\mathcal{G}}\norm{G_1-G_2}_{C^k([0,L_0])}.
\end{equation*}
This follows immediately from computing the derivatives of order up to $k$ with the chain rule and the product rule and Lemmas \ref{1lipschitzdependenceofthestressonthegrowthtensor} and \ref{1Lipschitzdependencenutrients}. Similarly, we can easily see that for all $G \in B_1$
\begin{equation*}
    \norm{\hat{\mathcal{G}}(G)}_{C^k([0,L_0])}=\sum_{l=0}^k \sup_{X\in[0,L_0]}|\partial_X^{(l)}(\mathcal{G}(G(X),S(G),N(G)(X),X))|\leq M_{\mathcal{G}}.
\end{equation*}
We obtain local existence of a solution $G\in C^1([0,\hat{T}^k_1], C^k([0,L_0]))$ by choosing $\hat{T}^k_1>0$ small enough for the inequalities $L_\mathcal{G}\hat{T}^k_1<1$ and $\hat{T}^k_1\leq \min\{\frac{1}{2M_0}G_{min}, \frac{R_k}{M_\mathcal{G}}\}$ to hold. Here, $M_\mathcal{G}$ denotes the uniform bound on $\mathcal{G}$ with respect to the $C^k$-norm on $B_1=B_0^k \cap B_1^0$ and $M_0$ denotes the uniform bound of $\mathcal{G}$ on $B_1^0$. According to Picard-Lindel\"of's theorem there exist a unique solution $G\in C^1([0,\hat{T}^k_1];C^k([0,L_0]))$ of the ordinary differential equation $G'=\hat{\mathcal{G}}(G)$ where $\hat{T}^k_1$ still depends on our choice of $R_k$. 

It remains to prove uniform a priori bounds. As in \cite{Bangert.Dolzmann.2023} we can show that $G(t)\in B_0^0$ for all $t\in [0,\hat{T}^0]$ with $\hat{T}^0$ chosen as in \cite{Bangert.Dolzmann.2023}, that is, $\hat{T}^0$ is the time-interval on which we obtain local existence of solutions of the ordinary differential equation on the space $C^1([0,T];C^0([0,T]))$ with Picard-Lindel\"of's theorem. It remains to show that $G(t)\in B_0^k$ for all $t\in [0,T]$. In the proof of Theorem 3 in\cite{Bangert.Dolzmann.2023} it was shown that we obtain a solution $G\in C^1([0,T];C^0([0,L_0])$ where $G\in B_0^0$ for all $t\in [0,T]$. Note, that at this point, it does not make a difference if we solve the ordinary differential equations for functions with values in $C^0([0,L_0])$ or $\mathcal{L}^{\infty}([0,L_0])$. The authors of \cite{Bangert.Dolzmann.2023} chose the latter approach. 

In the next step, we show global existence of a solution $G\in C^1([0,T]; C^1([0,L_0])$.
Because $G\in C^1([0,\hat{T}^k_1];C^k([0,L_0]))$ we can compute $\partial_X G(0,X)$. This will define $R_1$. We will show that $|\partial_X G(t,X)|$ is bounded, for all $X\in [0,L_0]$ and $t \in [0,\hat{T}^k_1]$, by the constant $R_1$ depending only on $\partial_X G(0,X)$ and the global constants which is yet to be determined. By differentiating the equation on the time interval $[0,\hat{T}^k_1]$ we obtain that for all $X\in [0,L_0]$
\begin{align*}
    \partial_t \partial_X G(t,X)=& \mathcal{G}_1(G(t,X),S(G)(t),N(G)(t,X),X)\partial_X G(t,X)\\&+ \mathcal{G}_3(G(t,X),S(G)(t),N(G)(t,X),X)\partial_X N(G)(t,X)\\&+\mathcal{G}_4(G(t,X),S(G)(t),N(G)(t,X),X)
\end{align*}
where $\mathcal{G}_i$ denotes the partial derivative with respect to the $i$-th variable. Now, we want to estimate $(\partial_X G(t,X))^2$ in order to estimate the absolute value of $\partial_X G(t,X)$. Because $G\in B_0^0$, Young's inequality asserts that for some $C>0$ depending only on the global constants
\begin{align*}
    \partial_t (\partial_X G(t,X))^2=& 2\partial_X G(t,X)\partial_t \partial_X G(t,X)\\=& 2(\mathcal{G}_1(G(t,X),S(G)(t),N(G)(t,X),X)\partial_X G(t,X)\\&+ \mathcal{G}_3(G(t,X),S(G)(t),N(G)(t,X),X)\partial_X N(G)(t,X)\\&+\mathcal{G}_4(G(t,X),S(G)(t),N(G)(t,X),X))\partial_X G(t,X)\\ \leq & C(\partial_X G(t,X))^2+C
\end{align*}
for all $t\in[0,\hat{T}^k_1]$. Here, we used that $G\mapsto N(G)=n_G\circ\phi_G\circ g$ and according to the chain rule and Lemmas \ref{boundednessofminimiserinCkplus1} and \ref{1regularityoncoefficients}, $\partial_X N(G)(t,X)$ is uniformly bounded on $B_1^0$. Hence, we can apply Gronwall's inequality and obtain that for all $t\in [0,\hat{T}^k_1]$
\begin{equation*}
    (\partial_X G(t,X))^2\leq e^{Ct}((\partial_X G(0,X))^2+Ct) \leq e^{CT}(\norm{\partial_X G(0,\cdot)}_{C^0([0,L_0])}^2+CT)=:R_1^2.
\end{equation*}
Therefore, $|\partial_X G(t,X)|$ is bounded by a constant depending only on the global constants and $\norm{\partial_X G(0,\cdot)}_{C^0([0,L_0])}$.

Now, we are able to prove global existence of solutions. Because $G(\hat{T}^k_1)\in B_0^1$, we may set $t_0=\hat{T}^k_1$, $\hat{T}^k_2=\min\{T,2\hat{T}^k_1\}$ and $G_0=G(\hat{T}^k_1)$ and argue as in the previous step to continue the solution and to obtain the existence of a solution in $C^1([0,\hat{T}^k_2],C^1([0,L_0]))$. As before, this solution satisfies $G(t)\in B_0^1$ for all $t \in [0,\hat{T}^k_2]$ and after finitely many steps we obtain a solution on $[0,T]$. Uniqueness of the solution follows from Picard-Lindel\"of's theorem.

Inductively, we show that for $k\geq 2$, we obtain a global solution $G\in C^1([0,T];C^k([0,L_0]))$. Assume, we have a global solution $G\in C^{1}([0,T];C^{l-1}([0,L_0]))$ for $2\leq l\leq k$. Because in particular $G\in C^1([0,\hat{T}^k_1];C^l([0,L_0]))$, we can compute $\partial_X^j G(0,X)$ for $1\leq j\leq l$. This will lead to the existence of a constant $R_l$ with $\norm{ G(t,\cdot)}_{C^l [0,L_0]}\leq R_l$ on $[0,T]$ where $R_l$ only depends on $\partial_X^j G(0,X)$ and the global constants. We have already seen by Picard-Lindel\"of's theorem, that we obtain a time horizon $\hat{T}^k_1$ for which a unique solution $G\in C^1([0,\hat{T}^k_1];C^l([0,L_0]))$ exists. We now show that $|\partial_X^l G(t,X)|$ is, for all $X \in [0,L_0]$, bounded by $R_l$ depending only on $\partial_X^j G(0,X)$ for $1\leq j\leq l$ and the global constants. In this step we determine an upper bound for the minimal value which $R_l$ is allowed to take. By differentiating the equation on the time interval $[0,\hat{T}^k_1]$ we obtain that for all $X\in [0,L_0]$
\begin{align*}
    \partial_t \partial_X^{l} G(t,X)=& \mathcal{G}_1(G(t,X),S(G)(t),N(G)(t,X),X)\partial_X^{l}G(t,X) \\&+ \sum_{j=2}^l\partial_X^{l-j}((\partial_X\mathcal{G}_1(G(t,X),S(G)(t),N(G)(t,X),X))\partial_X^{j-1}G(t,X))\\&+ \partial_X^{l-1}(\mathcal{G}_3(G(t,X),S(G)(t),N(G)(t,X),X)\partial_XN(G)(t,X))\\&+\partial_X^{l-1}\mathcal{G}_4(G(t,X),S(G)(t),N(G)(t,X),X).
\end{align*}

Now, we want to estimate $(\partial_X^l G(t,X))^2$ in order to estimate the absolute value of $\partial_X^l G(t,X)$. Because $G(t) \in B_0^{l-1}$ for all $t \in [0,\hat{T}^k_1]$, Young's inequality asserts that for some $C>0$ depending only on the global constants
\begin{equation*}
    \partial_t (\partial_X^l G(t,X))^2\leq  C(\partial_X^l G(t,X))^2+C
\end{equation*}
for all $t\in [0,\hat{T}^k_1]$. Here, we also used that $G\mapsto N(G)=n_G\circ\phi_G\circ g$ and according to (G3), the chain rule and Lemmas \ref{boundednessofminimiserinCkplus1} and \ref{1regularityoncoefficients}, $\partial_X^l N(G)$ is uniformly bounded on $B_1^{l-1}$. Therefore, we can apply Gronwall's inequality and obtain that for all $t\in [0,\hat{T}^k_1]$
\begin{equation*}
    (\partial_X^l G(t,X))^2  \leq e^{CT}((\partial_X^l G(0,X))^2+CT)=:R_l^2.
\end{equation*}
Hence, $|\partial_X^l G(t,X)|$ is bounded by a constant depending only on the global constants and $\partial_X^j G(0,X)$ on $[0,\hat{T}^k_1]$ with $1\leq j\leq l$.

Now, we are able to prove global existence of solutions. Because $G(\hat{T}^k_1)\in B_0^l$, we may, as before, set $t_0=\hat{T}^k_1$, $\hat{T}^k_2=\min\{T,2\hat{T}^k_1\}$ and $G_0=G(\hat{T}^k_1)$ and argue as in the previous step to continue the solution and to obtain the existence of a solution in $C^1([0,2\hat{T}^k_1],C^l([0,L_0]))$. As before, this solution satisfies $G(t)\in B_0^l$ for all $t$ and, after finitely many steps, we obtain a solution on $[0,T]$. Uniqueness of the solution follows from Picard-Lindel\"of's theorem. This proves the claim.
\end{proof}

\begin{example}
    As in \cite{Bangert.Dolzmann.2023} the following example satisfies the assumptions (G1)-(G4) and thus gives rise to a unique solution of the growth problem for the case $m=1$:
    \begin{equation*}
        \hat{\mathcal{G}}:\mathcal{C}^k_+([0,L_0])\mapsto\mathcal{C}^k([0,L_0]),\qquad \hat{\mathcal{G}}(G)(X)=\gamma(X)\mu(S(G))\eta(N(G)(X))G(X)
    \end{equation*}
    Here, $S(G)\in\RR$ and $N(G)\in H^{k+2}(0,L_0)$ denote the stress and the nutrient concentration induced by the growth via $G$. we assume that $\gamma\in C^k([0,L_0];[\gamma_0,\gamma_1])$, that $\mu\in W^{k+1,\infty}(\RR;[\mu_0,\mu_1])$ and that $\eta\in W^{k+1,\infty}((0,\infty);[\eta_0,\eta_1])$ with $\eta$ and $\nu$ increasing and $0<\gamma_0<\gamma_1$, $\mu_0<\,u_1$ and $0\leq\eta_0\leq\eta_1$. As in \cite{Bangert.Dolzmann.2023} we can find sub- and supersolutions that satisfy (G1) and (G4). It is easy to check that (G2) and (G3) are satisfied as well.
\end{example}

\section{Global Existence of Solutions  with Higher Regularity in Space and Time}

To prove additional time regularity of the solutions of the growth problem in the sense of Definition \ref{Definition Chapter Regularity Definition of Solutions} for the case with $m\geq 2$ arbitrary, cf. Theorem \ref{maintheoremsection1dtimespaceregularityhighertimeregularity}, it remains to show that for $1\leq l\leq m-1$ the map $t\mapsto S(t)$ is in the space $C^{l}([0,T];\mathbb{R})$ and that the map $t\mapsto N(\cdot,t)$ is in the space $C^{l}([0,T];C^k([0,L_0]))$ if $G\in C^{l}([0,T];C^k([0,L_0])$. In this section, we assume that $G\in C^1([0,T];C^k([0,L_0]))$ is a solution of the growth problem with $m=1$ in the sense of Definition \ref{Definition Chapter Regularity Definition of Solutions}. 

\begin{lemma}
    If $G\in C^{l}([0,T];C^k([0,L_0]))$ for $1\leq l\leq m-1$, then $g\in C^{l}([0,T];C^{k+1}([0,L_0]))$.
\end{lemma}

\begin{proof}
    By Definition $G(t,X)=\partial_X g(t,X)$ for all $X\in [0,L_0]$ and $t\in [0,T]$. The statement follows by the fundamental theorem of calculus.
\end{proof}

The following lemma states that the map $t\mapsto S(t)$ is sufficiently regular in time. The proof relies on the fact that by further differentiating the map $\pi_0$ as in the proof of Lemma \ref{1LipschitzCk} and taking into account assumption (W4), we attain as much regularity in the variable $S$ as $W_p$ allows us to. Then we can modify the proof of Lemma 6 in \cite{Bangert.Dolzmann.2023} to obtain higher regularity of the map $G\mapsto S(G)$. 

\begin{lemma}\label{additionaltimeregularitystresses1dsection1}
    If $G\in C^{l}([0,T];C^k([0,L_0]))$ for $1\leq l\leq m-1$, then $S\in C^{l}([0,T];\mathbb{R})$.  
\end{lemma}

\begin{proof}
    We can already deduce from Lemma \ref{1lipschitzdependenceofthestressonthegrowthtensor} that the map $t\mapsto S(t)$ is in the space $C^1([0,T];\mathbb{R})$ as a concatenation of the maps $t\mapsto G(t, \cdot)$, which is in the space $C^{l}([0,T];C^k([0,L_0]))$ by assumption, and $G(t, \cdot)\mapsto S(t)$, which is in the space $C^1(B_1,\mathbb{R})$ by Lemma \ref{1lipschitzdependenceofthestressonthegrowthtensor}. 
    By differentiating the right-hand side in equation \ref{section1additionaltimeregularitypi0} $(m-2)$-times and taking into account that $\partial_p W\in C^{k+m}([0,L_0]\times (0,\infty);\mathbb{R})$ and $\partial_{pp}W>0$ we immediately see that at most $m-2$ derivatives of $\partial_{pp} W$ appear in the numerator and the denominator is uniformly bounded from below. Hence, for all $S_0,S_1 \in \RR$ with $S_0 < S_1$, $\pi_0\in C^{m-1}([S_0,S_1];C^{k+1}([0,L_0]))$. Therefore, the map $\Phi$ as in the proof of Lemma 6 in \cite{Bangert.Dolzmann.2023} is in the space $C^{m-1}((\Sigma_0-1,\Sigma_1+1)\times B_1; \mathbb{R})$. As in Lemma 6 in \cite{Bangert.Dolzmann.2023}, we define $\Phi:\RR \times C^k([0,L_0])\to \RR$ via
    \begin{equation*}
        (S,G)\mapsto \Phi(S,G)= \int_0^{L_0} \pi_0(X,S)G(X)dX-l_0
    \end{equation*}
    and for each $G\in B_1$ there exist exactly one $S$ with $S=S(G)$ such that $\Phi(S(G),G)=0$. Moreover, for $\Delta G \in C^{k}([0,L_0])$
    \begin{equation*}
        \frac{\partial\Phi}{\partial G}(S,G)[\Delta G]=\int_0^{L_0}\pi_0(X,S)\Delta G(X)dX
    \end{equation*}
    and $S \in \RR$
    \begin{equation*}
        \frac{\partial\Phi}{\partial S}(S,G)[\Delta S]=\biggl(\int_0^{L_0}\frac{1}{W_{pp} (X,\pi_0(X,S))}G(X)dX\biggr)\Delta S.
    \end{equation*}
    We notice that $\frac{\partial\Phi}{\partial G}(S,G)$ does not depend on $G$ anymore and that the derivative of $\Phi$ with respect to $S$ is linear in $G$. Furthermore, we can compute $(\frac{\partial}{\partial S})^{m-1}\Phi$ and see that derivatives of order at most $m$ of $W_p$ and $m-2$ of $\pi_0$ in the variable $S$ appear. Exemplarily, we compute $(\frac{\partial}{\partial S})^{2}\partial\Phi$. 
    \begin{align*}
        &\biggl(\frac{\partial}{\partial S}\biggr)^{2}\Phi(S,G)[\Delta S]=\frac{\partial}{\partial S}\biggl(\int_0^{L_0}\frac{1}{W_{pp} (X,\pi_0(X,S))}G(X)dX\biggr)[\Delta S]\\=&\biggl(\int_0^{L_0} \frac{d}{d\varepsilon}\biggr|_{\varepsilon=0} \frac{1}{\partial_{pp} W(X,\pi_0(X,S+\varepsilon\Delta S))}G(X)dX\biggr)\Delta S\\=&\biggl(\int_0^{L_0}-\frac{\partial_{ppp}W(X,\pi_0(X,S))\partial_S \pi_0(X,S)\Delta S}{(\partial_{pp}W(X,\pi_0(X,S)))^2}G(X)dX\biggr) \Delta S.\\=&\biggl(\int_0^{L_0}-\frac{\partial_{ppp}W(X,\pi_0(X,S))\partial_S \pi_0(X,S)}{(\partial_{pp}W(X,\pi_0(X,S)))^2}G(X)dX\biggr) (\Delta S)^2.
    \end{align*}
    As we have seen before, $\pi_0\in C^{m-1}([S_0,S_1];C^{k+1}([0,L_0]))$ and therefore, all derivatives appearing in $(\frac{\partial}{\partial S})^{m-1}\Phi$ are contained in a compact set. Moreover, as we have seen in the proof of proposition \ref{maintheoremsection1dtimespaceregularity}, $G(t, \cdot)\in B_1$ for all $t\in [0,T]$. Therefore, $(\frac{\partial}{\partial S})^{m-1}\Phi$ is uniformly bounded on $[\Sigma_0-1, \Sigma_1+1]\times B_1$ by a constant depending only on the global constants. As in the proof of Lemma 6 in \cite{Bangert.Dolzmann.2023} we can apply the implicit function theorem. Note, that in particular, as in the proof of Lemma 6 in \cite{Bangert.Dolzmann.2023}, we obtain that $\frac{\partial\Phi}{\partial S}$ is bounded from below by a positive constant. Due to the higher regularity of the map $\Phi$, the map $G\mapsto S(G)$ is in the space $C^{m-1}(B_1,\mathbb{R})$. Because the map $t\mapsto G(t, \cdot)$ is in the space $C^{l}([0,T];C^k([0,L_0]))$ by assumption, the concatenation $t\mapsto S(t)$ is in the space $C^{l}([0,T];\mathbb{R})$.
\end{proof}

Because the nutrient concentration in the reference configuration $N(t,X)$ with $X\in [0,L_0]$ and $t\in [0,T]$ is defined via $N(t,X)=n(y(t,X),t)$, we have to show that the map $t\mapsto y(\cdot, t)$ lies in the space $C^{m-1}([0,T];C^k([0,L_0]))$. Actually, the map $t\mapsto y(\cdot, t)$ will even be in the space $C^{m-1}([0,T];C^{k+1}([0,L_0]))$. This is a consequence of Lemma \ref{additionaltimeregularitystresses1dsection1}.

\begin{corollary}\label{additionaltimeregularitytotaldeformation1dsection1}
    If $G\in C^{l}([0,T];C^k([0,L_0]))$ for $1\leq l\leq m-1$, the map $t\mapsto y(t, \cdot)$ is in the space $C^{l}([0,T];C^{k+1}([0,L_0]))$. Furthermore, assuming $G\in C^{l}([0,T];C^k([0,L_0]))$, the map $(\partial_X\phi)\circ g$ is in the space $C^{l}([0,T];C^{k+1}([0,L_0]))$.
\end{corollary}

\begin{proof}
    By Definition $\pi_0(X,S(t))=(\partial_X\phi)(t, g(t, X))$. We have already seen in the proof of Lemma \ref{additionaltimeregularitystresses1dsection1} that $\pi_0\circ S$ and therefore $(\partial_X\phi)\circ g$ is in the space $C^{l}([0,T];C^{k+1}([0,L_0]))$. Because the map $t\mapsto \partial_X g(t, \cdot)$ is in the space $C^{l}([0,T];C^k([0,L_0]))$ the product $t\mapsto (\partial_X \phi)(t, g(t, \cdot))\partial_X g(t, \cdot)$ is in the space $C^{l}([0,T];C^k([0,L_0]))$. Hence, the map $t\mapsto \partial_X(\phi(t, g(t, X)))$ is in the space $C^{l}([0,T];C^k([0,L_0]))$ and the map $t\mapsto y(t, \cdot)$ is in the space $C^{m-1}([0,T];C^{k+1}([0,L_0]))$.
\end{proof}

To prove additional time-regularity of the solution of the elliptic reaction-diffusion equation, we first recall the following standard result on time-dependent bounded linear operators on Banach spaces.

\begin{lemma}\label{timeregularityinverseellipticoperatorssection1}
    Let $X$,$Y$ be Banach spaces and $A\in C^l([0,T];\mathcal{L}(X,Y))$ for $l \in \NN$ such that $A(t)^{-1}$ exists for all $t\in [0,T]$ and $\norm{A(t)^{-1}}_{\mathcal{L}(Y,X)}\leq C$ for all $t\in [0,T]$. Then the inverse is also differentiable in time, that is, $A^{-1}\in C^l([0,T];\mathcal{L}(Y,X))$ and
    \begin{equation*}
        \frac{d}{dt}A(t)^{-1}=-A(t)^{-1}A'(t)A(t)^{-1}.
    \end{equation*}
\end{lemma}

The following Lemma states that the map $t\mapsto N(t, \cdot)=n(t, y(t, \cdot))$ is sufficiently regular in time. 

\begin{lemma}\label{additionaltimeregularitynutrients1dsection1}
    If $G\in C^{l}([0,T];C^k([0,L_0]))$ for $1\leq l\leq m-1$, the map $t\mapsto N(t, \cdot)$ is in the space $C^{l}([0,T];C^k([0,L_0]))$.
\end{lemma}

\begin{proof}
    We have seen in Corollary \ref{additionaltimeregularitytotaldeformation1dsection1} that the map $t\mapsto y(t, \cdot)$ is in the space $C^{l}([0,T];C^{k+1}([0,L_0]))$ if $G\in C^{l}([0,T];C^k([0,L_0]))$. It remains to show that the map $t\mapsto n(t, \cdot)$, where $n$ denotes the solution of the elliptic reaction-diffusion equation as in Lemma \ref{1regularityoncoefficients}, is in the space $C^l([0,T];C^{k}([0,l_0]))$. We do this by applying Lemma \ref{timeregularityinverseellipticoperatorssection1}. We first note that the maps $D$ and $\beta$ as defined in Lemma \ref{regularitycoefficients1d} are in the space $C^l([0,T];C^k([0,l_0]))$. This can be easily seen by writing 
    \begin{align*}
        D_G&=(\phi_G'\circ \phi_G^{-1})(D_0\circ y_G^{-1})\\&=(\phi_G'\circ g \circ g^{-1} \circ\phi_G^{-1})(D_0\circ y_G^{-1})\\&=((\phi_G'\circ g) \circ  y_G^{-1})(D_0\circ y_G^{-1})
    \end{align*}
    and
    \begin{align*}
        \beta_G&=(\phi_G'\circ \phi_G^{-1})(\beta_0\circ y_G^{-1})\\&=(\phi_G'\circ g \circ g^{-1}\circ \phi_G^{-1})(\beta_0\circ y_G^{-1})\\&=((\phi_G'\circ g) \circ  y_G^{-1})(\beta_0\circ y_G^{-1}).
    \end{align*}
    According to our assumption that $D_0,\beta_0\in W^{k+m}([0,L_0])$ it follows from Lemma \ref{regularitycoefficients1d} and Corollary \ref{additionaltimeregularitytotaldeformation1dsection1} that the maps $D$ and $\beta$ are in the space $C^{l}([0,T];C^k([0,l_0]))$ as a concatenation of sufficiently regular maps. Note that in particular the map $y_G^{-1}$ is in the space $C^{l}([0,T];C^{k+1}([0,L_0]))$, which follows from the invertability of the map $y_G$ and Corollary \ref{additionaltimeregularitytotaldeformation1dsection1}. We set $X=H^{k+1}([0,l_0])\cap H_0^1([0,l_0])$ and $Y=H^{k-1}([0,l_0])$. By elliptic regularity theory and the standard Sobolev embeddings, the inverse of the map $A \in C^l([0,T];\mathcal{L}(X,Y))$ defined via $A(t):\hat{n}( \cdot)\mapsto -\partial_x(D(t,\cdot)\partial_x \hat{n}( \cdot))+\beta(t,\cdot)\hat{n}(\cdot)$ is a linear and continuous map from $Y$ to $X$ and the operator norm is uniformly bounded in $t\in [0,T]$ by standard elliptic regularity estimates. Moreover, because the maps $D$ and $\beta$ are in the space $C^{l}([0,T];C^k([0,l_0]))$, the operator $A$ is in the space $C^l([0,T];\mathcal{L}(X,Y))$. If we define $h(t,x):= \partial_x(D(t,x)\frac{n_R-n_L}{l_0})- \beta(t,x)(\frac{n_R-n_L}{l_0}x+ n_L)$, then $\hat{n}(t,\cdot) \in X$ satisfies
    \begin{equation*}
        -\partial_x(D(t,\cdot)\partial_x \hat{n}(t,\cdot))+\beta(t,\cdot)\hat{n}(t,\cdot)=h(t,\cdot)
    \end{equation*}
    with $\hat{n}(t,0)=\hat{n}(l_0)=0$ if and only if $n(t,x):= \hat{n}(t,x) + \frac{n_R-n_L}{l_0}x + n_L$ satisfies $n(t,0)=n_L$, $n(t,l_0)=n_R$ and
    \begin{equation*}
        -\partial_x(D(t,\cdot)\partial_x n(t,\cdot))+\beta(t,\cdot)(t,\cdot)=0.
    \end{equation*}
    By definition $n(t, \cdot) \in H^{k+1}([0,l_0])$. Hence, Lemma \ref{timeregularityinverseellipticoperatorssection1} proves the claim by writing
    \begin{equation*}
        n(t,x):=A(t)^{-1}[h(t, \cdot)](x)+\frac{n_R-n_L}{l_0}x+n_L,
    \end{equation*}
    which implies that $n \in C^l([0,T]; H^{k+1}([0,l_0])) \subset C^l([0,T]; C^{k}([0,l_0]))$ by the standard Sobolev embeddings in one spatial dimension.
\end{proof}

\begin{theorem}\label{maintheoremsection1dtimespaceregularityhighertimeregularity}
    Suppose that $T>0$, that $W$ satisfies (W1)-(W4) and that $\hat{\mathcal{G}}$ satisfies (G1)-(G4). Then, there exist a unique solution of the growth problem in the sense of Definition \ref{Definition Chapter Regularity Definition of Solutions} on $[0,T]$.
\end{theorem}

\begin{proof}
    According to Theorem \ref{maintheoremsection1dtimespaceregularity}, there exists a global solution of the growth problem in the sense of Definition \ref{Definition Chapter Regularity Definition of Solutions} on $[0,T]$ in the space $C^1([0,T];C^k([0,L_0]))$. By integrating the ordinary differential equation, we inductively show $C^m$-regularity of the solution in time. Hence, we must use that the maps $t\mapsto S(t)$ with values in $\mathbb{R}$ and $t\mapsto N(t, \cdot)$ with values in $C^k([0,L_0])$ are $(m-1)$-times differentiable with respect to $t$. Then, we can deduce $C^m$-regularity in time. We integrate and write
\begin{equation}\label{additionaltimeregularity1dBochnerintegralformulation}
    G(t, X)=G(0,X)+\int_0^t \mathcal{G}(G(t, X),S(t),N(t, X),X)dt.
\end{equation}
Because we have shown in Lemma \ref{additionaltimeregularitystresses1dsection1} and in Lemma \ref{additionaltimeregularitynutrients1dsection1} that for $m\geq 2$ and $1\leq l\leq m-1$ the maps $S\in C^{l}([0,T];\mathbb{R})$ and $N\in C^{l}([0,T];C^k([0,L_0]))$ if $G\in C^{l}([0,T];C^k([0,L_0]))$, we can deduce, by taking into account assumption (G2) and applying Lemma \ref{additionaltimeregularitystresses1dsection1} and Lemma \ref{additionaltimeregularitynutrients1dsection1}, that the integrand in equation \ref{additionaltimeregularity1dBochnerintegralformulation} is in the space $C^{l}([0,T];C^k([0,L_0]))$. By the fundamental theorem of calculus $G\in C^{l+1}([0,T];C^k([0,L_0]))$. Note in particular that the fundamental theorem of calculus also accounts for differentiability at the boundary of the interval $[0,T]$ where the differential is to be understood one-sidedly. Inductively, we obtain that $G\in C^{m}([0,T];C^k([0,L_0]))$.
\end{proof}

\vspace{\baselineskip}

\bibliographystyle{alpha}
\bibliography{regularityelasticrods.bib}

\end{document}